\documentclass[11pt]{amsart}
\usepackage{geometry}                
\usepackage[parfill]{parskip}    
\usepackage{amssymb, amsthm, amsmath, mathrsfs}
\usepackage{color}
\usepackage{soul}
\usepackage{url}
\usepackage{pinlabel}
\usepackage{caption}
\usepackage{subcaption}

\theoremstyle{plain}
\newtheorem{theorem}{Theorem}

\newtheorem{claim}[theorem]{Claim}

\newtheorem{conjecture}[theorem]{Conjecture}
\newtheorem{corollary}[theorem]{Corollary}

\newtheorem{lemma}[theorem]{Lemma}

\newtheorem{definition}[theorem]{Definition}

\newtheorem{observation}[theorem]{Observation}

\theoremstyle{definition}
\newtheorem{fact}[theorem]{Fact}
\newtheorem{remark}[theorem]{Remark}

\newcommand{\F}{\mathbb{F}}
\newcommand{\Z}{\mathbb{Z}}

\newcommand{\x}{\mathbf{x}}
\newcommand{\y}{\mathbf{y}}

\newcommand{\tree}{\mathscr{T}}
\newcommand{\xmin}{\tilde{\mathbf{x}}}

\newcommand{\hfk}{\widehat{HFK}}

\DeclareMathOperator{\rank}{rank}

\title{Pretzel knots with $L$-space surgeries}

\author[Lidman]{Tye Lidman}
\address{Department of Mathematics \\
          The University of Texas\\
          Austin, TX 78712, USA
         }
\email{tlid@math.utexas.edu}

\author[Moore]{Allison Moore}
\address{Department of Mathematics \\
          The University of Texas\\
         Austin, TX 78712, USA
        }
\email{moorea8@math.utexas.edu}

\begin{document}
\maketitle
\begin{abstract}
A rational homology sphere whose Heegaard Floer homology is the same as that of a lens space is called an $L$-space.  We classify pretzel knots with any number of tangles which admit $L$-space surgeries.  This rests on Gabai's classification of fibered pretzel links.  \end{abstract}

\section{Introduction}

The Heegaard Floer homology of three-manifolds and its refinement for knots, knot Floer homology, have proved to be particularly useful for studying Dehn surgery questions in three-manifold topology.  Recall that the knot Floer homology of a knot $K$ in the three-sphere is a bigraded abelian group, 
\[
\widehat{HFK}(K) = \oplus_{m,s} \widehat{HFK}_m(K,s),
\] 
introduced by Ozsv\'ath and Szab\'o \cite{OS:KnotInvariants} and independently by Rasmussen \cite{Rasmussen:Thesis}.  The graded Euler characteristic is the symmetrized Alexander polynomial of $K$~\cite{OS:KnotInvariants},
\[
\Delta_K(t) = \sum_{s} \chi (\hfk(K, s)) \cdot t^s.
\]
These theories have been especially useful for studying lens space surgeries.  For example, if $K \subset S^3$ admits a lens space surgery, then for all $s \in \mathbb{Z}$, we have $\widehat{HFK}(K,s) \cong  0$ or $\mathbb{Z}$ \cite[Theorem 1.2]{OS:LensSpaceSurgeries}.  Knot Floer homology detects both the genus of $K$ by $g(K) = \max \{ s \mid \widehat{HFK}(K,s) \neq 0\}$ \cite{OS:HolDisksGenusBounds} and the fiberedness of $K$, by whether $\hfk(K,g(K))$ is isomorphic to $\Z$ \cite{Ghiggini:Fibered, Ni:Fibred}. Together, this implies that a knot in $S^3$ with a lens space surgery is fibered. 
Indeed, this result applies more generally to knots in $S^3$ admitting $L$-space surgeries.  Recall that a rational homology sphere $Y$ is an \emph{$L$-space} if $|H_1(Y;\mathbb{Z})| = \rank \widehat{HF}(Y)$, where $\widehat{HF}$ is the ``hat'' flavor of Heegaard Floer homology. The class of $L$-spaces includes all lens spaces, and more generally, three-manifolds with elliptic geometry \cite[Proposition 2.3]{OS:LensSpaceSurgeries} (or equivalently, with finite fundamental group by the Geometrization Theorem \cite{KL2008}).  A knot admitting an $L$-space surgery is called an {\em $L$-space knot}.  



The goal of this paper is to classify $L$-space pretzel knots\footnote{Throughout this paper, we use the convention that pretzel knots are prime.  It is known to experts that $L$-space knots are prime, so there is no loss of generality with this assumption.}.  For notation, we use $(n_1,\ldots,n_r)$ to denote the pretzel knot with $r$ tangles, where the $i$th tangle consists of $n_i \in \Z$ half-twists.  We use $T(a,b)$ to denote the $(a,b)$-torus knot.      
\begin{theorem}\label{thm:pretzels}
Let $K$ be a pretzel knot.  Then, $K$ admits an $L$-space surgery if and only if $K$ is isotopic to $ \pm (-2,3,q)$ for odd $q \geq 1$ or $T(2,2n+1)$ for some $n$.  
\end{theorem}

We first remark that the pretzel knots $(-2,3,1) $, $(-2,3,3)$ and $(-2,3,5)$ are isotopic to the torus knots $T(2,5)$, $T(3,4)$, and $T(3,5)$, respectively. In general, torus knots are well-known to admit lens space surgeries \cite{Mos:TorusKnots}; the hyperbolic pretzel knot $(-2,3,7)$ is also known to have two lens space surgeries \cite{FS:LensSpaces}.  The knot $(-2,3,9)$ has two finite, non-cyclic surgeries \cite{BH:SphericalSpaceForms}.  Finally, the remaining knots, $(-2,3,q)$ for $q \geq 11$, are known to have Seifert fibered $L$-space surgeries with infinite fundamental group \cite{OS:LensSpaceSurgeries}.  Therefore, in this paper we show that no other pretzel knot admits an $L$-space surgery.  This will be proved by appealing to Gabai's classification of fibered pretzel links \cite{Gabai:Detecting} and the state-sum formula for the Alexander polynomial \cite{Kauffman:Formal, OS:Alternating}.  

Using Theorem~\ref{thm:pretzels}, we are able to easily recover the classification of pretzel knots which admit surgeries with finite fundamental group due to Ichihara and Jong.  

\begin{corollary}[Ichihara-Jong \cite{ID:CyclicFinite}]\label{cor:finitefillings}
The only non-trivial pretzel knots which admit non-trivial finite surgeries up to mirroring are $(-2,3,7)$, $(-2,3,9)$, $T(3,4)$, $T(3,5)$, and $T(2,2n+1)$ for $n > 0$. 
\end{corollary}
\begin{proof}
As discussed above, the knots in the statement of the corollary are known to admit finite surgeries.  Therefore, it remains to rule out the case of $(-2,3,q)$ for odd $q \geq11$. Using the theory of character varieties, Mattman proved that the only knots of the form $K = (-2,3,q)$ with $q \neq 1,3,5$ which admit a finite surgery are $(-2,3,7)$ and $(-2,3,9)$  \cite{Mattman:CSPretzels}.  This completes the proof.  
\end{proof}

\begin{remark}
In fact, Ichihara and Jong show Corollary~\ref{cor:finitefillings} holds more generally for Montesinos knots \cite{ID:CyclicFinite}. Their proof uses similar arguments to those in this paper, but first appeals to an analysis of essential laminations on the exteriors of Montesinos knots by Delman.  This allows them to restrict their attention to a few specific families of pretzel knots before reducing to the case of the $(-2,3,q)$-pretzel knots.    
\end{remark}

Finally we observe that while many pretzel knots have essential Conway spheres, the pretzel knots with $L$-space surgeries do not.  We conjecture that this holds for $L$-space knots in general. If true, this fact would imply that an $L$-space knot admits no nontrivial mutations.
\begin{conjecture}\label{conj:lspaceconwaysphere}
If $K$ is an $L$-space knot, then there are no essential Conway spheres in the complement of $K$.  
\end{conjecture}

\section*{Acknowledgments} We would like to thank Cameron Gordon, Jennifer Hom, and Matthew Hedden for helpful discussions.  We would also like to thank Liam Watson for suggesting we check Conjecture~\ref{conj:lspaceconwaysphere} on such a family of knots. The first author acknowledges the partial support of NSF RTG grant DMS-0636643 and the second author would like to acknowledge the support of NSF RTG grant DMS-1148490.
\section{Background}

Throughout, $K$ (resp. $L$) is an oriented knot (resp. link) in $S^3$. Let $g(K)$ denote the genus of $K$.  Let $L=(n_1,\dots, n_r)$ be a pretzel link. We will also use the integer $n_i$ to refer to this specific tangle in the pretzel projection, where $|n_i|$ is the \emph{length} of the tangle $n_i$. Notice that tangles of length one can be permuted to any spot in a pretzel link by flype moves. Furthermore, if there exist $n_i=+1$ and $n_j=-1$ in $L$, then $n_i$ and $n_j$ can be pairwise removed by flyping followed by an isotopy. Unless otherwise stated, we assume any diagram of a pretzel link $L$ is in pretzel form and that $r$ is the minimal possible number of strands to present $L$ as a pretzel projection. Note this implies that there do not exist indices $i$ and $j$ such that $n_i = \pm1$ and $n_j = \mp 2$. Throughout, we will implicitly assume the classification of pretzel knots due to Kawauchi~\cite{Kawauchi:Pretzels}. 
 
\subsection{Determinants of pretzel knots}
\label{subsec:det}
Since $\chi(\widehat{HFK}(K,s)) = a_s$, the coefficient of $t^s$ in the symmetrized Alexander polynomial of $K$, this will give us an easy way to approach Theorem~\ref{thm:pretzels} in many cases; whenever there exists a coefficient $a_s$ of $\Delta_K(t)$ with $|a_s|>1$, $K$ is not an $L$-space knot \cite{OS:LensSpaceSurgeries}.
We therefore establish the following lemma.
\begin{lemma}
\label{lem:det}
If $det(K)>2g(K)+1$ then $\Delta_K(t)$ contains some coefficient $a_s$ with $|a_s|>1$.
\end{lemma}
\begin{proof}
If the coefficients of $\Delta_K(t)$ are at most one in absolute value, then
\[ 
	\det(K) =| \Delta_K(-1)| \leq \sum_s |a_s| \leq 2g(K)+1. \qedhere
\]
\end{proof}
Suppose that $Y$ is a Seifert fibered rational homology sphere with base orbifold $S^2$ and Seifert invariants $(b; (a_1,b_1), \dots (a_r,b_r))$. Then
\[
| H_1(Y;\mathbb{Z}) | = | a_1\cdots a_r \cdot (b + \sum_{i=1}^{r} \frac{b_i}{a_i}) |
\]
(see for instance~\cite{Saveliev:Invariants}). The branched double covers of Montesinos knots (and consequently, pretzel knots) are such Seifert fibered spaces. 
If $K =(n_1,\dots, n_k, \underbrace{1,\dots, 1}_{d})$, where $|n_i|>1$ for $1\leq i \leq k$, then $\Sigma_2(K)$ has Seifert invariants $(d; (n_1, 1), \dots, (n_k,1))$. Therefore, 
\begin{equation}
\label{eqn:det}
	 det(K) = |H_1(\Sigma_2(K))|  = \big| n_1\cdots n_k \cdot (d +\sum_{i=1}^k \frac{1}{n_i} ) \big|.
\end{equation}
As permuting tangles in a pretzel knot corresponds with doing a series of Conway mutations, $\Delta_K(t)$, and consequently $det(K)$, are unchanged. Invariance of the determinant under permutations is also evident from Equation~\ref{eqn:det}. Since the symmetrized version of the Alexander polynomial of a fibered knot is monic of degree $g(K)$, when $K$ is fibered and the mutation preserves fiberedness, the genus of $K$ is also unchanged.
\subsection{Fibered pretzel links}
As mentioned earlier, if $K$ is an $L$-space knot then $K$ is fibered. Theorem~\ref{thm:pretzels} is therefore automatic for any non-fibered knot. Thus for the proof of Theorem~\ref{thm:pretzels} we will only be interested in fibered pretzel knots.  In~\cite[Theorem 6.7]{Gabai:Detecting}, Gabai classified oriented fibered pretzel links together with their fibers; we recall this below in Theorem~\ref{thm:fiberedpretzels}. An oriented pretzel link $L$ may be written
\[
	L=\left( m_1, m_{11}, m_{12}, \dots, m_{1\ell_1} ,m_2, m_{21},\dots, m_{2\ell_2},  \dots m_R, m_{R1},\dots, m_{R\ell_R} \right) ,
\]
where $m_i$ denotes a tangle in which the two strands are oriented consistently (i.e. both up or both down) and $m_{ij}$ denotes a tangle where the two strands are oriented inconsistently (i.e. one up and one down). 
An oriented pretzel link falls into one of three types which can be easily ascertained from a diagram: a Type 1 link contains no $m_i$, a Type 2 link contains both an $m_i$ and an $m_{ij}$, and a Type 3 link contains no $m_{ij}$. Moreover, associated to a Type 2 or Type 3 link $L$ will be an auxiliary oriented pretzel link $L'$,
\begin{figure}
	\centering
      	\includegraphics[height=35mm]{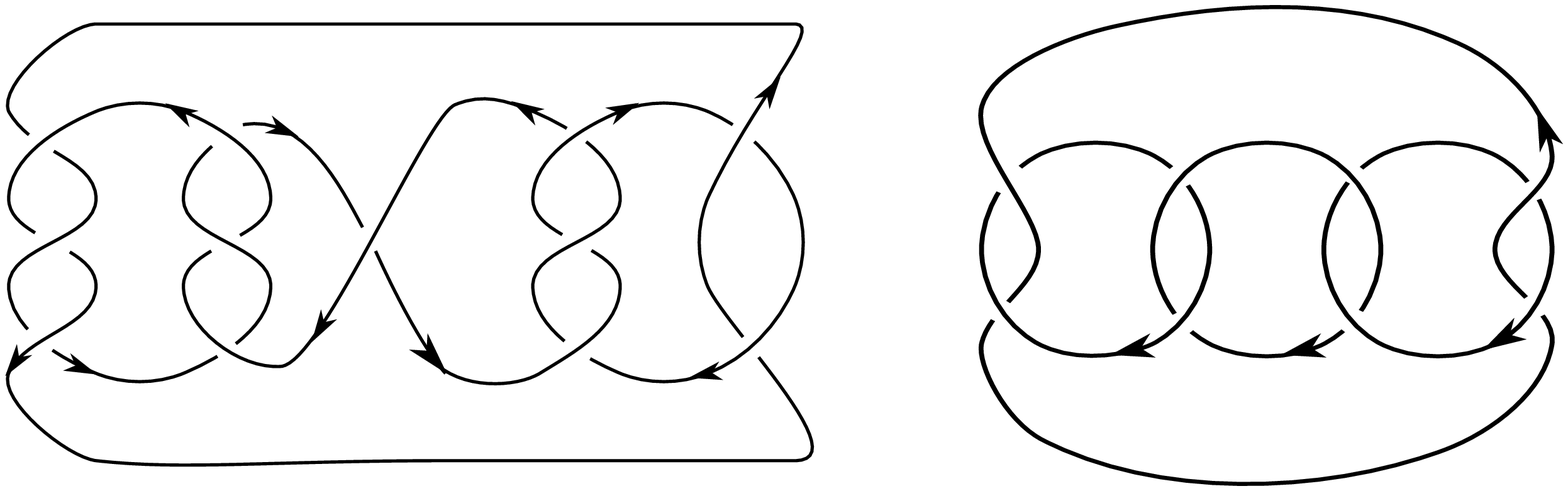}
          \caption{The pretzel knot $(3,-3,1,3,2)$ and its associated auxiliary link $(-2,2,-2,2)$.}
          \label{fig:auxiliary2}
\end{figure}
\begin{eqnarray}
\label{eqn:L'}
\begin{split}
	L' = \left( \frac{-2m_1}{|m_1|}, m_{11}, m_{12}, \dots, m_{1\ell_1} , \frac{-2m_2}{|m_2|}, m_{21},\dots, m_{2\ell_2}, \dots  \right. \\
		\left. \dots , \frac{-2m_R}{|m_R|}, m_{R1},\dots, m_{R\ell_R} \right) ,
\end{split}
\end{eqnarray}
where the term $\frac{-2m_i}{|m_i|}$ is omitted if $|m_i|=1$. The link $L'$ is oriented so that the surface obtained by applying the Seifert algorithm is of Type 1. See Figure~\ref{fig:auxiliary2}. The auxiliary link $L'$ is derived from a procedure of Gabai in which a minimal genus Seifert surface is desummed and its sutured manifold hierarchy is analyzed to determine whether $L$ fibers~\cite{Gabai:Detecting}.
\begin{theorem}[Gabai, Theorem 6.7 in \cite{Gabai:Detecting}] 
\label{thm:fiberedpretzels}
The algorithm which follows determines whether an oriented pretzel link fibers.\footnote{The original formulation describes the fiber surfaces for all types; we include this information only when it is relevant to our calculations.}

{\bf Algorithm.} A pretzel link $L$ is one of three types.

{\bf Type 1:} Then $L$ fibers if and only if one of the following holds:
\begin{enumerate}
	\item each $n_i=\pm1$ or $\mp3$ and some $n_i=\pm1$.
	\item $(n_1, \dots, n_r) = \pm(2,-2,2, -2,\dots, 2,-2,n)$, $n\in\Z$ (here, $r$ is odd).
	\item $(n_1, \dots, n_r) = \pm(2,-2,2, -2,\dots, -2,2,-4)$ (here, $r$ is even).
\end{enumerate}

{\bf Type 2:} Fibered Type 2 links fall into the following three subcases.
\begin{itemize}
	\item[] {\bf Type 2A:} The numbers of positive and negative $m_i$ differ by two. Then $L$ fibers if and only if $|m_{ij}| = 2$ for all indices $ij$.
	\item[] {\bf Type 2B:} The numbers of positive and negative $m_i$ in $L$ are equal and $L'\neq$ $\pm(2,-2,\dots,2,-2)$. Then $L$ fibers if and only if $L'$ fibers.
\end{itemize}
{\bf Type 3:} If either the numbers of positive and negative tangles are unequal or if $L'\neq \pm(2,-2,\dots, 2, -2)$, then treat $L$ as if it was Type 2A or 2B. Otherwise, $L$ is fibered if and only if there is a unique $m_i$ of minimal absolute value.  

Finally, if $L$ is a fibered pretzel link of Type 1, Type 2A, or the Type 2A subcase of Type 3, then the fiber surface is necessarily the surface obtained by applying the Seifert algorithm to the pretzel diagram of $L$.  
\end{theorem}

In our case analysis, we denote the three subcases of Type 3 by Type 3-2A, Type 3-2B, and Type 3-min accordingly.

\begin{remark}
In Gabai's classification of oriented fibered pretzel links, there is a third subcase of fibered Type 2 links, called Type 2C. For these links, the numbers of positive and negative $m_i$ are equal and $L'=\pm(2, -2, \dots,2,-2)$. However, these links are not minimally presented and can be isotoped to be in Type 3. 
\end{remark}

\begin{remark}
If a pretzel knot $K$ (as opposed to a link) is Type 1, there is an odd number of $m_{ij}$, all of which are odd. If $K$ is Type 2, there is exactly one $m_{ij}$, which we denote by $\bar{m}$, and this unique $\bar{m}$ must also be the unique even tangle. Moreover, there is an even number of $m_i$. If $K$ is Type 3, there is an even number of $m_i$, exactly one of which is even.
\end{remark}

\subsection{A state sum for the Alexander polynomial}
\label{sec:statesum}
The Alexander polynomial of $K$ admits a state sum expression in terms of the set of Kauffman states $\mathcal{S}$ of a decorated projection of the knot~\cite{Kauffman:Formal}. We will use a reformulation of the Kauffman state sum which appears in~\cite{OS:Alternating}. By a decorated knot projection we mean a knot projection with a distinguished edge. When using decorated knot projections, we will always choose the bottom-most edge in a standard projection of a pretzel knot to be the distinguished edge. Each state $\x$ is equipped with a bigrading $(A(\x), M(\x))\in \Z\oplus \Z$ such that the symmetrized Alexander polynomial of $K$ is given by the state sum
\begin{equation}
\label{eqn:euler}
	\Delta_K(t) = \sum_{\x\in\mathcal{S}} (-1)^{M(\x)}t^{A(\x)}.
\end{equation}
Let $G_B$ and $G_W$ denote the black and white graphs associated with a checkerboard coloring of a decorated knot projection. The decorated edge of $K$ determines a decorated vertex, the \emph{root}, in each of $G_B$ and $G_W$. For a pretzel diagram, there is also clearly a top-most vertex of $G_B$, referred to as the \emph{top vertex}. The set of states $\mathcal{S}$ is in a one-to-one correspondence with the set of maximal trees of $G_B$. Each maximal tree $T\subset G_B$ uniquely determines a maximal tree $T^*\subset G_W$. Fix a state $\x\in\mathcal{S}$ and let $\tree_\x=T_\x \cup T^*_\x$ denote the black and white maximal trees which correspond to $\x$. By an abuse of notation, we will not always distinguish between the state $\x$ and the trees $\tree_\x$. We now describe $A(\x)$ and $M(\x)$ in this framework, following~\cite{OS:Alternating}.  
\begin{figure}[h]
	\labellist
	\small\hair 2pt
	\pinlabel $e^*$ at 82 100
	\pinlabel $e$ at 70 125
	\pinlabel (above)\;\;$\eta(e)=-1,$ at 30 75
	\pinlabel $\eta(e^*)=0$ at 82 75
	\pinlabel (above)\;\;$\eta(e)=+1,$ at 140 75
	\pinlabel $\eta(e^*)=0$ at 192 75
	\pinlabel (below)\;\;$\eta(e)=0,$ at 30 55
	\pinlabel $\eta(e^*)=+1$ at 82 55
	\pinlabel (below)\;\;$\eta(e)=0,$ at 140 55
	\pinlabel $\eta(e^*)=-1$ at 192 55
	\endlabellist
	\centering
	\includegraphics[height=65mm]{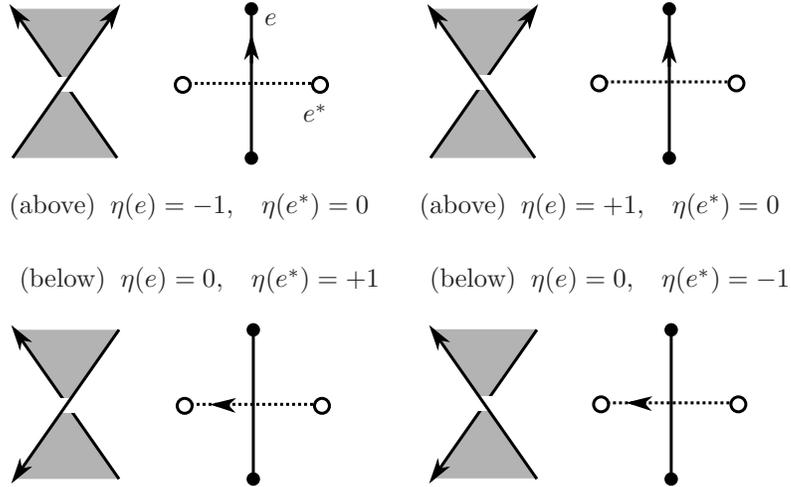}
 	\caption{The labels $\eta(e)$ and $\eta(e^*)$ for the edges $e\in G_B$ and $e^*\in G_W$. The edge orientations pictured are those induced by $K$ on $G_B$ or $G_W$.}
    	\label{fig:labels}
\end{figure}
Label each edge $e$ of $G_B$ and $G_W$ with $\eta(e)\in\{-1,0,1\}$ according to Figure~\ref{fig:labels}. We describe two partial orientations on the edges of $T_\x$ and $T^*_\x$. The first orientation is a total orientation which flows away from the root. The second partial orientation is induced by the orientation on the knot as in Figure~\ref{fig:labels}; note that at each crossing exactly one of the edges of $T_\x$ or $T^*_\x$ is oriented. Then, $A(\x)$ is defined by
\begin{equation}
\label{eqn:Agrading}
	A(\x) = \frac{1}{2} \sum_{e \in \tree_\x} \sigma(e) \eta(e), 
\end{equation}
where
	\begin{equation*}
	  	\sigma(e) = \left\{
	    	\begin{array}{rl}
		0 & \text{if $e$ is not oriented by $K$} \\
		+1 & \text{if the two induced orientations on $e$ agree} \\
		-1 & \text{if the two induced orientations on $e$ disagree.} 
	    	\end{array} \right.
	\end{equation*}
Note that though it is not indicated in the notation, $\sigma(e)$ depends on $\x$, and which $\x$ will be clear from the context; $\eta(e)$ does not depend on $\x$. Next, $M(\x)$ is defined by summing only over edges on which the two orientations agree,
\begin{equation}
\label{eqn:Mgrading}
	M(\x) = \sum_{	
		\begin{tiny}
		\begin{array}{c}
			e \in \tree_\x \\
			\sigma(e)=+1
		\end{array}
		\end{tiny}
	} \eta(e).
\end{equation}
An example of a state and its bigrading is given in Figure~\ref{fig:comboexample}. \begin{figure}
        \begin{subfigure}[b]{.99\textwidth}
		\labellist
		\tiny\hair 1pt
		\endlabellist
   		\centering
                	\includegraphics[height=40mm]{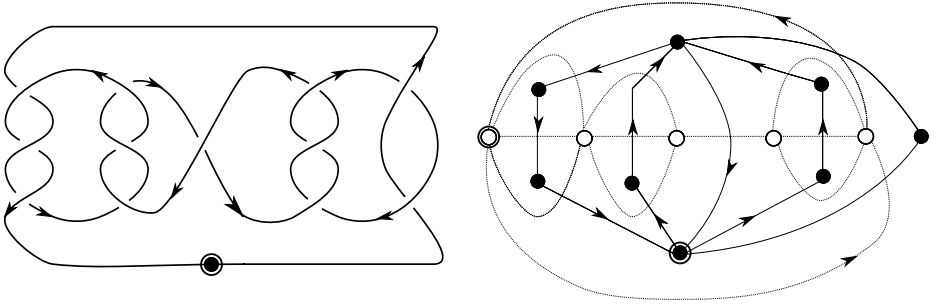}
                \caption{The Type 2A knot $K=(3,-3,1,3,2)$ and the corresponding graphs $G_B$ and $G_W$ with orientations induced by $K$ and black and white roots indicated. 
                }
                \label{fig:comboexample3}
        \end{subfigure}\\
           
        \begin{subfigure}[b]{.99\textwidth}
        		\labellist
		\footnotesize\hair 1pt
		\pinlabel $+1$ at 315 195
		\pinlabel $-1$ at 95 180
		\pinlabel $-1$ at 35 95
		\pinlabel $0$ at 35 30
		\pinlabel $+1$ at 165 160
		\pinlabel $+1$ at 145 120
		\pinlabel $0$ at 135 30
		\pinlabel $0$ at 195 90
		\pinlabel $-1$ at 260 40
		\pinlabel $-1$ at 305 125
		\pinlabel $-1$ at 235 185				
		\pinlabel $0$ at 195 90		
		\pinlabel $0$ at 345 55
		\endlabellist
		\centering
                	\includegraphics[height=40mm]{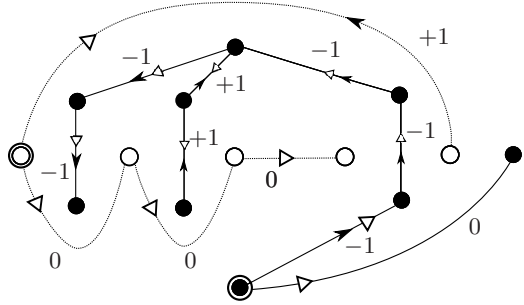}
                \caption{A state $\x$ of $K$ in bigrading $(A(\x), M(\x)) = (-4, -5).$ White arrows indicate the orientations which point away from roots and black arrows indicate the orientations induced by $K$. Edges are labeled by $\eta$.}
                \label{fig:comboexample4}
        \end{subfigure}    
        \caption{An example to illustrate $G_B$ and $G_W$ for the pretzel knot $(3,-3,1,3,2)$ and the bigrading corresponding to a state.} \label{fig:comboexample}
\end{figure}
\subsection{Counting lemmas}
\label{subsec:countinglemmas}
The state-sum formula (Equation~\ref{eqn:euler} above) provides an elementary way to determine the coefficients of the Alexander polynomial. Suppose that the state-sum decomposition of a diagram of a fibered knot $K$ admits a unique state $\xmin$ with minimal $A$-grading $A(\xmin)$. Since the symmetrized Alexander polynomial is monic of degree $g(K)$, then $A(\xmin)=-g(K)$ by Equation~\ref{eqn:euler}. When such a unique minimal element $\xmin$ exists, it is convenient to use $\xmin$ to count the states in $A$-grading $-g(K)+1$. We will often exploit this to show that $|a_{-g(K)+1}|>1$, demonstrating that many pretzel knots are not $L$-space knots.  

\begin{definition}
\label{def:trunk} 
Let $K$ be a pretzel knot with a decorated diagram and let $\tree_\x$ be the trees corresponding to some state $\x$. The \emph{trunk} of $T_\x$ (or just $\x$) is the unique path in $T_\x$ which connects the root of $G_B$ to the top vertex of $G_B$ (see Figure~\ref{fig:comboexample4}). 
\end{definition}
Each tangle $n_i$ determines a path in $G_B$ from the root to the top vertex; let $T(n_i)$ denote this path. We collect the following facts to use freely throughout without reference. 
\begin{fact}Let $\x$ be any state and let $\xmin$ be the unique minimally $A$-graded state, if it exists.
\begin{enumerate}
	\item The trunk of $T_\x$ is necessarily $T(n_k)$ for some $k$. If $i\neq k$, $T(n_i)\cap T_\x\neq T(n_i)$. 
	\item If $|n_i|=1$ and $T(n_i)$ is not the trunk of $T_{\x}$, then $T(n_i)\cap T_{\x}=\emptyset$.
	\item For any $i$, $\eta$ is constant along the edges in $T(n_i)$.
	\item When $T(m_i)$ is not the trunk of the unique minimally $A$-graded state $\xmin$, there is only one terminal edge in $T_{\tilde{x}}\cap T(m_i)$. In particular, $T(m_j)\cap T_{\xmin}$ is connected and cannot have edges incident to both the top vertex and the root. 
\end{enumerate}
\end{fact}
\begin{definition}
\label{def:trades}
Let $K$ be a pretzel knot with a decorated diagram and suppose there exists a unique state $\xmin$ with minimal $A$-grading. Fix a tangle $n_i\neq \pm1$ which does not correspond to the trunk. A \emph{trade} is a state $\y$ (or $\tree_\y$) whose corresponding black tree is obtained by replacing the terminal edge of $T_{\xmin}$ contained in $T(n_i)$ with the unique edge in $T(n_i) \smallsetminus T_{\xmin}$. See Figure~\ref{fig:tradenottrade}. 
\end{definition}
In a trade, $T_{\y}$ (resp. $T^*_{\y}$) along with its orientations and labels differs from $T_{\xmin}$ (resp. $T^*_{\xmin}$) in exactly one edge, and furthermore, $T_{\y}$ and $T_{\xmin}$ share the same trunk. 
\begin{lemma}
\label{lem:trades}
Suppose that $K=(n_1,\dots, n_r)$ and $\xmin$ are as in Definition~\ref{def:trades} and that $T(n_k)$ is the trunk of $\xmin$.  Let $\ell$ be the number of tangles with $n_i = \pm 1$ and $i\neq k$. Then, there are $r-\ell-1$ trades, all of which are supported in bigrading $(A(\xmin)+1,M(\xmin) +1)$. 
\end{lemma}
\begin{figure}[htb]
	\labellist
	\small\hair 2pt
	\endlabellist
	\centering
	\includegraphics[height=27mm]{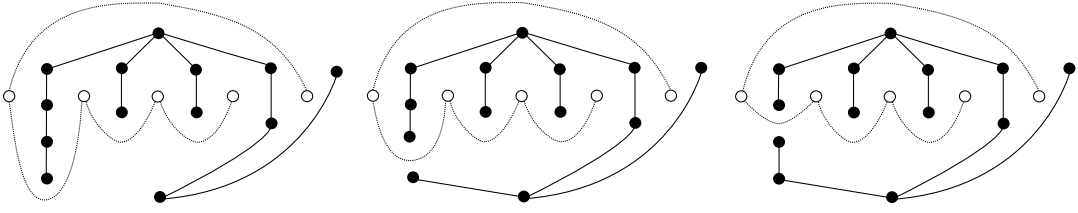}
 	\caption{Three states for the pretzel knot $(5,-3,3,3,2)$. If the knot is oriented so that the strands of the first tangle point downward, the first state is the unique state with minimal $A$-grading, the middle state is a trade and the last state is neither.}
    	\label{fig:tradenottrade}
\end{figure}
\begin{proof}
Let $\y$ be a trade. By definition, there is exactly one trade corresponding with each tangle of length greater than one which is not the trunk (see Figure~\ref{fig:tradenottrade}), and so there are $r - \ell -1$ trades. Let $e_{\tilde{\x}} \in T_{\xmin}$ and $e_\y \in T_\y$ ($e^*_{\tilde{\x}}$ and $e^*_\y$ respectively) be the edges along which $T_{\xmin}$ and $T_\y$ ($T^*_{\xmin}$ and $T^*_\y$ respectively) differ. The edges $e_{\tilde{\x}}$ and $e_\y$ are contained in some $T(n_i), i\neq k$, and therefore share the same value for $\eta$. Assume first that $\eta(e_{\tilde{\x}})=\eta(e_\y)=\pm1$ and $\eta(e^*_{\tilde{\x}})=\eta(e^*_\y)=0$. Because $A(\xmin)$ is minimal and $\xmin$ is unique, $\sigma(e_{\tilde{\x}})\eta(e_{\tilde{\x}})=-1$, or else $A(\y) \leq A(\xmin)$.  This implies $\sigma(e_{\tilde{\x}})=-\eta(e_{\tilde{\x}}).$ In the trade, $e_{\xmin}$ is replaced with $e_\y$ and the orientations induced by the root on $T_{\xmin}$ and $T_\y$ switch from pointing down on $e_{\xmin}$ to pointing up on $e_\y$ (or vice versa). Hence $\sigma(e_\y)= -\sigma(e_{\tilde{\x}})$. This implies $\sigma(e_\y)\eta(e_\y)=+1$, and therefore both $M(\y)=M(\xmin)+1$ and $A(\y) = A(\xmin) + 1$. Assume next that $\eta(e^*_{\tilde{\x}})=\eta(e^*_\y)=\pm1$ and $\eta(e)=\eta(e_\y)=0$. The trade induces a change in $T^*_{\xmin}$ wherein the edge $e^*_{\xmin}$ is replaced with an edge $e^*_\y$ which is vertically adjacent in $G_W$ (see Figure~\ref{fig:tradenottrade}). Similarly, since $A(\xmin)$ is minimal $\sigma(e^*_{\tilde{\x}})\eta(e^*_{\tilde{\x}})=-1$. The same argument as for $G_B$ applies and we obtain $M(\y)=M(\xmin)+1$ and $A(\y) = A(\xmin) + 1$. 
\end{proof}

For the remainder of the paper, we proceed through the cases of Theorem~\ref{thm:fiberedpretzels} to prove Theorem~\ref{thm:pretzels}. In all cases (exempting the two families of knots mentioned in Theorem~\ref{thm:pretzels}), for each fibered knot $K$ we will exhibit an Alexander grading $s$ where $\hfk(K,s)$ is neither trivial nor isomorphic to $\mathbb{Z}$.  As discussed, this implies these knots are not $L$-space knots. For most fibered pretzel knots, we will do this by showing that there is a coefficient of the Alexander polynomial with $|a_s|>1$. Except for a few sporadic knots, we accomplish this by making repeated use of two basic arguments: either studying $a_{-g(K)+1}$ with the state-sum formula or by analyzing the determinant of $K$ and applying Lemma~\ref{lem:det}. In fact the Alexander polynomial serves as an obstruction for all but one knot. We will show:
\begin{observation}
Up to mirroring, there is a unique fibered pretzel knot which has the Alexander polynomial of an $L$-space knot which does not admit an $L$-space surgery. This knot is $(3,-5,3,-2)$.
\end{observation}
Before proceeding, we point out that pretzel knots with one strand are unknotted and that the two stranded pretzel $(a,b) \simeq T(2,a+b)$. In all of the cases which follow, $K$ is a minimally presented fibered pretzel knot with three or more tangles, unless otherwise stated.
\section{Type 1 Knots}
\label{sec:type1}
We will only need Lemma~\ref{lem:det} to determine which Type 1 pretzel knots are $L$-space knots.
\begin{lemma}
\label{lem:det-rs}
The only $L$-space pretzel knots of Type 1 are those isotopic to the $T(2, 2n+1)$ torus knots. Any other fibered pretzel knot $K$ of Type 1 satisfies $\det (K) > 2g(K)+1$.
\end{lemma}
\begin{proof}
In our case analysis, we disregard the sub-cases (2) and (3) of Type 1 because these are links with at least two components. Thus up to mirroring, 
\[
	K = (\underbrace{1,\dots, 1}_{c}, \underbrace{-3,\dots, -3}_{d}),
\]
where $c>0$ and $d\geq0$. When $d=0$, $K$ is the torus knot $T(2,c)$. Thus assume $d>0$.  If $K$ has three strands, then $K$ is isotopic to either $(1,-3,-3)$ or $(1,1,-3)$, which are $T(2,3)$ and the figure eight knot, respectively. The figure eight knot has $det(K) = 5>2g(K)+1$. Therefore, we may assume that $K$ has at least four strands (in fact five, since if $K$ is a Type 1 knot, it must have an odd number of strands).  More generally, the genus of the pretzel spanning surface (and in this case, the genus of $K$ by Theorem \ref{thm:fiberedpretzels}) is given by 
\[
	g(K)=\frac{1}{2}(d+c-1).
\]
By Equation~\ref{eqn:det},
\[
	det(K) = | 3^d(-c + \sum_{i=1}^d \frac{1}{3} ) | = |3^{d-1} (d-3c) |.
\]
We will verify the inequality in two cases, $d>3c$ and $d<3c$, where $c, d>0$ and $d+c\geq 5$. (When $d=3c$, $d+c$ is even and so $K$ is not a knot.) If $d>3c$, then
\[
	det(K) =|3^{d-1} (d-3c) | \geq |3^{d-1}|  > \frac{4d}{3} > d+c =2g(K)+1 .
\]
Consider $d<3c$. If $d<3$, the inequality is easily checked by hand. If $3\leq d<3c$, we have
\begin{eqnarray*}
	3^{d-1} -1  > 2d  &\Rightarrow& (3^{d-1}-1)(3c-d) > 2d-2c \\
	&\Rightarrow &(3^{d-1}-1)(3c-d) +(3c-d) > d+c \\
	&\Rightarrow &det(K) = 3^{d-1}(3c-d)  > d+c =2g(K)+1. \qedhere
\end{eqnarray*} 
\end{proof}
\section{Type 2 knots}
\label{sec:type2}
We remind the reader that a Type 2 knot has an odd number of tangles and contains exactly one $m_{ij}$, which is even and denoted $\bar{m}$.
\subsection{Type 2A} 
\label{subsec:type2a}
After mirroring, we may assume that a Type 2A fibered knot has $p+2$ positive odd tangles, $p$ negative odd tangles, and $\bar{m} = \pm2$. The proof of Theorem \ref{thm:pretzels} for Type 2A knots is addressed via Lemmas \ref{lem:type2a-1}, \ref{lem:type2a-2} and \ref{lem:type2a-3}.

\begin{lemma}
\label{lem:type2a-1}
Up to mirroring, the only $L$-space pretzel knots of Type 2A with three tangles are those isotopic to $(-2,3,q)$, for $q\geq1$ odd. Otherwise there exists a coefficient $a_s$ of $\Delta_K(t)$ such that $|a_s|\geq 2$.
\end{lemma}
\begin{proof}
Here $K=(\pm2,r,q)$, minimally presented, where $r$ and $q$ are positive, odd integers. For $K=(2,r,q)$, $K$ is alternating and hyperbolic, hence not an $L$-space knot~\cite{OS:LensSpaceSurgeries}. Therefore, we may assume $K=(-2,r,q)$, with $r>1$. When $r= 3$ and $q$ is any positive odd integer, this is the family of $L$-space knots exempted in the assumptions of the lemma.  

Without loss of generality, we may further assume that $5\leq r \leq q$.  The genus of the surface $F$ obtained by applying the Seifert algorithm to the pretzel presentation for $K=(-2,r,q)$ is $g(F)=\frac{1}{2}(r+q)$, which is equal to $g(K)$ by Theorem~\ref{thm:fiberedpretzels}. Thus, whenever $r >5$  and $q>5$ or whenever $r = 5$ and $q>7$, 
\begin{eqnarray*}
det(K) &=& |2rq(\frac{1}{r} + \frac{1}{q} - \frac{1}{2})| = |2(r + q) - rq| > r+q+1 = 2g(K)+1.
\end{eqnarray*}
It remains to check $r=5$ and $q=5$ or $7$. We obtain the desired result by computing the Alexander polynomials\footnote{All Alexander polynomials in this paper are computed using the Mathematica package KnotTheory~\cite{knotatlas}.}:
\begin{eqnarray*}
&\Delta_{P(-2,5,5)}(t) =& t^{-5}-t^{-4}+ t^{-2}-2t^{-1}+ 3-2t+ t^2-t^4+ t^5 \\
&\Delta_{P(-2,5,7)}(t) =& t^{-6}-t^{-5} +t^{-3} -2t^{-2} +3t^{-1}- 3 + 3t - 2t^2 + t^3- t^5 +t^6.
\end{eqnarray*}\qedhere
 \end{proof}
\begin{lemma}
\label{lem:type2a-2}
Let $K=(n_1,\dots, n_{2p+3})$ be a fibered pretzel knot of Type 2A with $p\geq1$ and where there exists some tangle with $n_i<-2$. Then $|a_{-g(K)+1}|\geq 2$.
\end{lemma}
\begin{proof}
The condition of being a Type 2A fibered knot is preserved under permutation of tangles. As mentioned in Section~\ref{subsec:det}, the genus and $\Delta_K(t)$ are also preserved. Therefore, we may apply mutations to assume that $n_i$ is positive when $i$ is odd and $n_i$ is negative when $i$ is even, except for $n_{2p+2} >0$ and $n_{2p+3}=\bar{m} = \pm2$. Thus for all edges $e\in T(n_i)\subset G_B$,
\[
	\eta(e) =\left\{\begin{array}{ll}
	0 & \text{if }i = 2p+3 \\
	-1& \text{if } i<2p+3 \text{ is odd or } i=2p+2 \\ 
	+1 & \text{if } i\neq 2p+2  \text{ is even}. 
	\end{array}\right. 
\]
\begin{claim}
\label{claim:minunique1}
Orient $K$ so that the strands of the first tangle point downward. Then $K$ admits a unique state $\xmin$ with minimal $A$-grading.
\end{claim}
\begin{proof}[Proof of Claim~\ref{claim:minunique1}]
Let $\xmin$ be the state defined as follows and illustrated by the example in Figure~\ref{fig:type2a}. The trunk of $\xmin$ is $T(n_{2p+2})$. The intersections $T_{\xmin}$ with $T(n_i)$ for $i=1, \dots, 2p+1$ are incident to the top vertex, and therefore are not incident to the black root. There is a single edge in $T_{\xmin}\cap T(n_{2p+3})$ which is incident to the root if $\bar{m}=2$ or incident to the top vertex if $\bar{m}=-2$. 

By choice of the orientation on $K$, $T(n_i)$ is oriented downward for $i$ odd, and upward for $i$ even, except for $T(n_{2p+3})$, where instead the corresponding edges of $T^*_{\xmin}$ are oriented. In $T_{\xmin}$, the orientation induced by the root points downward along all $T(n_i)$, $i<2p+2$, and points upward along the trunk. 
Hence for all $e\in T_{\xmin}$,
\[
	\sigma(e) =\left\{\begin{array}{ll}
		 0 & \text{if } e\in T(n_{2p+3})\\
		+1& \text{if } e\in T(n_i), \text{for $i$ odd and } i\neq 2p+3\text{ or } i=2p+2\\
		-1 & \text{if } e\in T(n_i), \text{for $i$ even and } i\neq 2p+2.
	\end{array}\right.
\]
As for edges in the white tree $T^*_{\xmin}$, all are labeled $\eta(e)=\sigma(e)=0$ except for the one edge $\tilde{e}^*$ corresponding with $n_{2p+3}=\bar{m}$, which is labeled $\eta(\tilde{e}^*)=\pm1$ when $\bar{m}=\pm2$. In particular, the maximal tree with minimal $A$-grading is constructed so that $\sigma(\tilde{e}^*)\eta(\tilde{e}^*)=-1$ regardless of the sign of $\bar{m}$. See Figure \ref{fig:type2a}. Thus, every edge of $\tree_{\xmin}$ with $\eta(e)\neq0$ contributes $\sigma(e)\eta(e)=-1$ to the sum for $A(\xmin)$. 

We show that $A(\xmin)$ is minimal and $\xmin$ is unique. Fix an arbitrary state $\x$. Because there is exactly one edge $e^*\in T^*_\x$ labeled $\eta(e^*)\neq 0$ then,
\[
	A(\x) = \frac{1}{2} \big( \sigma(e^*)\eta(e^*) + \sum_{e \in T_\x} \sigma(e) \eta(e) \big).
\] 
In particular, for $\xmin$, 
\[
	A(\xmin) = \frac{1}{2} \big(-1 +  \sum_{e \in T_{\xmin}} \sigma(e) \eta(e) \big) .
\]
Suppose $\x$ is a state with the same trunk as $\xmin$ but for which $T_\x$ differs from $T_{\xmin}$ along any set of edges of $T(n_i)$, $i=1,\dots, n_{2p+1}$. Then there exists some edge of $T_\x$ which is incident to the root and this edge will contribute $\sigma(e)\eta(e)=+1$ to the sum for $A(\x)$. Since the contribution of the white tree $T^*_\x$ is not impacted, $A(\x)>A(\xmin)$. If instead $\x$ shares the same trunk as $\xmin$ but $T_\x$ differs from $T_{\xmin}$ along $T(n_{2p+3})$, then the edge $e^*\in T^*_\x$ will contribute $\sigma(e^*)\eta(e^*)=+1$ to the sum for $A(\x)$, and again $A(\x)>A(\xmin)$. Now, suppose $\x$ has a different trunk from $\xmin$. If the trunk of $\x$ is $T(n_{2p+3})$ and $T(n_i) \cap T_{\x}$ agrees with $T(n_i)\cap T_{\xmin}$ for $i=1,\dots, 2p+1$, then $A(\x)=A(\xmin)+1$. If the trunk of $\x$ is $T(n_{2p+3})$ and $T(n_i) \cap T_{\x}$ does not agrees with $T(n_i)\cap T_{\xmin}$, then $A(\x)>A(\xmin)+1$. If instead the trunk of $\x$ is $T(n_i)$ for some $i=1,\dots 2p+1$, then certainly $A(\x)\geq A(\xmin)+1$. Hence $A(\xmin)$ is minimal, and moreover, it follows from the above discussion that $\xmin$ is unique.
 \end{proof}
\begin{figure}[h]
        	\labellist
	\footnotesize\hair 1pt
	\pinlabel $\tilde{e}^*$ at 305 205
	\pinlabel $+1$ at 325 175
	\pinlabel $-1$ at 95 180
	\pinlabel $-1$ at 35 95
	\pinlabel $0$ at 35 30
	\pinlabel $+1$ at 165 160
	\pinlabel $+1$ at 145 120
	\pinlabel $0$ at 135 30
	\pinlabel $0$ at 195 90
	\pinlabel $-1$ at 215 45
	\pinlabel $-1$ at 270 120
	\pinlabel $-1$ at 215 160				
	\pinlabel $0$ at 195 90		
	\pinlabel $0$ at 305 55
	\endlabellist
	\centering
                	\includegraphics[height=45mm]{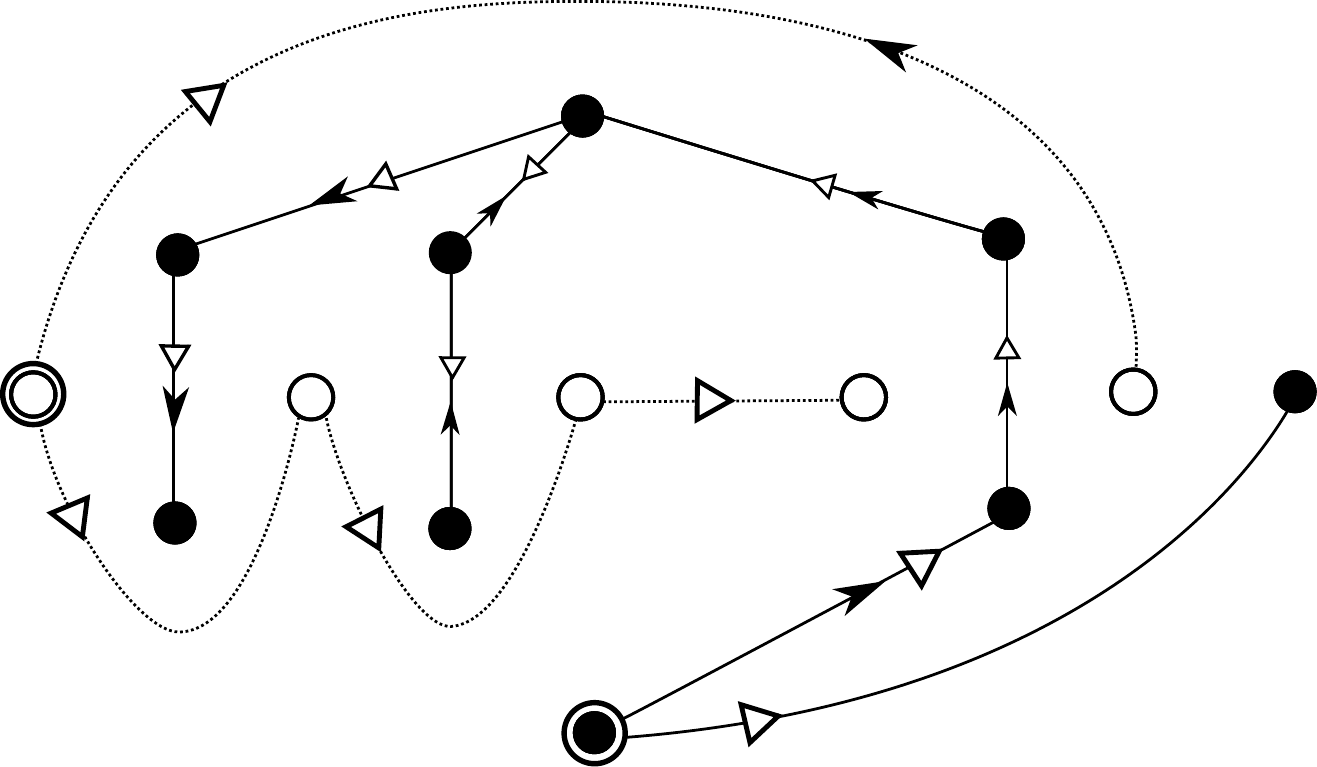}
	\caption{The trees $\tree_{\xmin}$ corresponding with the unique minimal state $\xmin$ of the Type 2A fibered knot $K=(3,-3,1,3,2)$. Edges in the diagram are labeled by $\eta$ and $\tilde{e}^*$ is indicated.}
    	\label{fig:type2a}
\end{figure}
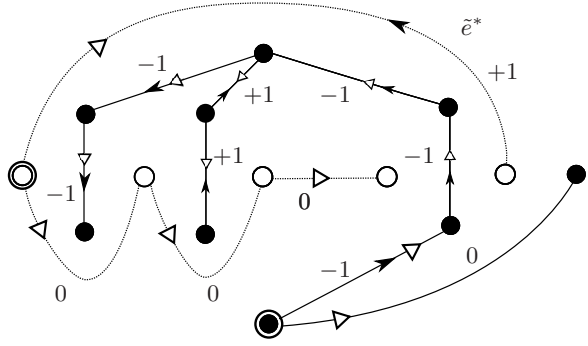
Let $\ell$ be the number of length one tangles excluding the trunk. By Lemma~\ref{lem:trades}, there are $2p+2-\ell$ trades, all supported in bigradings $(-g(K)+1, M(\xmin) +1)$. To determine that $|a_{-g(K)+1}|\geq 2$, we need to count the other states in $A$-grading $-g(K)+1$ and compute their $M$-gradings. 

Because $\bar{m}=\pm2$, all of the trees which share the same trunk as $T_{\xmin}$ which are not trades represent states which have an $A$-grading greater than $-g(K)+1$\footnote{We remark that when $|\bar{m}|>2$, there exist states in $A$-grading $-g(K)+1$ which arise from configurations other than trades or trees with new trunks.}. Thus, the states in $A$-grading $-g(K)+1$ which are not trades are states with different trunks. One of these states is denoted $\x'$, where $T_{\x'}$ differs from $T_{\xmin}$ only as follows. The trunk of $\x'$ is $T(n_{2p+3})$ and $T_{\x'}\cap T(n_{2p+2})$ is incident to the root. If $\bar{m}=-2$, $\x'$ is supported in bigrading $(-g(K)+1,M(\xmin)+2)$ and if $\bar{m}=+2$, $\x'$ is supported in bigrading $(-g(K)+1,M(\xmin)+1)$. Each remaining state in $A$-grading $-g(K)+1$ corresponds with a state denoted $\x_j$, where $T_{\x'}$ differs from $T_{\xmin}$ only as follows. The trunk of $\x_j$ is $T(n_j)$ for some $n_j=\pm 1$, $j \neq 2p+2$, and $T(\x_j)\cap T(n_{2p+2})$ is incident to the root. The trunk of $T_{\x_j}$ is necessarily length one because otherwise $A(\x_j)>-g(K)+1$ due to the contribution of at least two edges labeled $\sigma(e)\eta(e)=+1$ in $T(n_j)$.
\begin{claim}
\label{claim:mgradings1}
Let $\x_j$ be as above. Then,
\[
	M(\x_j) =\left\{\begin{array}{ll}
	M(\xmin) +1 & j \text{ odd and } j \neq 2p+3\\
	M(\xmin) +2 & j \text{ even and } j\neq 2p+2.
	\end{array}\right.
\]
\end{claim}
\begin{proof}[Proof of Claim~\ref{claim:mgradings1}]
In $T(n_{2p+2})$, all edges are labelled $\eta(e)=-1$. For all $e\in T_{\xmin}\cap T(n_{2p+2}) $,  $\sigma(e)=+1$. Because $n_j=\pm1$, $T_{\xmin}\cap T(n_j)=\emptyset$. Now $T_{\x_j}\cap T(n_{2p+2})$ contains $n_{2p+2}-1$ edges, all with $\sigma(e)=+1$. For the single edge $e\in T(n_j)\cap T_{\x_j}$, $\sigma(e)=\eta(e)=-1$ if $j$ is odd and $\sigma(e)=\eta(e)=+1$ if $j$ is even. All other edges and labels of $T_{\x_j}$ and $T_{\xmin}$ agree and the changes in the white graphs do not affect the $M$-grading.  The net change to the $M$-grading from $\xmin$ to $\x_j$ is $+1$ or $+2$, respectively. 
\end{proof}
By Equation \ref{eqn:euler}, the coefficient $|a_{-g(K)+1}|$ is given by the absolute value of the difference in the numbers of states in $M$-gradings $M(\xmin)+1$ and $M(\xmin)+2$. Suppose first that $\bar{m}=2$. Since $K$ is minimally presented, there are no $j$ with $n_j=-1$. Thus we may assume any tangle of length one is positive, and therefore all states in $A$-grading $-g(K)+1$ are supported in $M$-grading $M(\xmin)+1$. This implies $|a_{-g(K)+1}|>1$, since clearly there is more than one such state. Suppose now that $\bar{m}=-2$. We may similarly assume each length one tangle is negative. Since $n_{2p+2}>0$, the trunk is not length one and therefore $\ell$ is the number of length one tangles. By Lemma~\ref{lem:trades} and Claim~\ref{claim:mgradings1},
\[
	|a_{-g(K)+1}| = (2p+2-\ell)-(\ell +1) =2p-2\ell +1,
\]
and so $|a_{-g(K)+1}| > 1$ whenever $p>\ell$. When $p=\ell$, then every negative tangle other than $\bar{m}$ is length one. In other words, whenever there exists some tangle with $n_i < -2$, then $|a_{-g(K)+1}|\geq 2$. This verifies the statement of the lemma.
\end{proof}
In light of Lemmas \ref{lem:type2a-1} and \ref{lem:type2a-2}, after isotopy and our assumptions on mirroring,
\[
	K = (-2, \underbrace{-1,\dots, -1}_{p},w_1, w_2, \dots, w_{p+2}  )
\]  
where $w_i\geq3$ is odd for $1\leq i \leq p+2$ and $p\geq1$.
\begin{lemma}
\label{lem:type2a-3}
Let $K$ be as above. Then $det(K)>2g(K)+1$.
\end{lemma}
\begin{proof}
Since $K$ is a Type 2A fibered knot, then by Theorem~\ref{thm:fiberedpretzels}, the minimal genus Seifert surface and the fiber for $K$ is  obtained by applying the Seifert algorithm to the standard projection. This gives 
\begin{equation*}
	g(K) = \frac{1}{2} \big( \sum^{p+2}_{i=1}(w_i-1) + 2 \big).
\end{equation*}
Let $W=w_1 \dots w_{p+2}$. By Equation~\ref{eqn:det} and the fact that $w_i\geq3$ is odd for $1\leq i \leq p+2$,
\begin{eqnarray*}
	det(K) &=& \big| -2W \big(-p-\frac{1}{2}+ \sum^{p+2}_{i=1}\frac{1}{w_i} \big) \big| \\
	&=& \big| W + 2W(p - \sum^{p+2}_{i=1}\frac{1}{w_i} ) \big| \\
	&\geq&  \big| W + 2W ( \frac{2p-2}{3} )  \big|. \\
\end{eqnarray*}
Since $p\geq 1$, 
\begin{eqnarray*}	
	\big| W + 2W ( \frac{2p-2}{3} )  \big| &\geq&  W  \\
	&>&(\sum^{p+2}_{i=1} w_i ) + 1\\
	&\geq& \sum^{p+2}_{i=1} (w_i-1) + 3\\
	&=& 2 g(K) + 1 . \hspace{2cm} \qedhere
\end{eqnarray*}
\end{proof}
\subsection{Type 2B}
\label{subsection:type2b}
A Type 2B fibered pretzel knot $K$ has exactly one even tangle $\bar{m}$, which is the unique $m_{ij}$, $p$ positive odd tangles, and $p$ negative odd tangles, where $p\geq 1$. The auxiliary link $L'\neq \pm(2,-2,\dots, 2, -2)$, and $K$ fibers if and only if $L'$ fibers (see Equation~\ref{eqn:L'} for the construction of $L'$).  
\begin{lemma}
\label{lem:type2b}
For all minimally presented fibered pretzel knots of Type 2B, $|a_{-g(K)+1}|\geq 2$.
\end{lemma}
\begin{proof}
Suppose first that $K$ has $c>0$ length one tangles. Recall length one tangles do not factor into $L'$. If $c\geq3$, then $L'$ is not a Type 1 fibered links (see Theorem \ref{thm:fiberedpretzels}). If $c = 2$, the fiberedness of $L'$ implies $\bar{m}=\pm2$ when the length one tangles are $\mp 1$, and this is not allowed because $K$ is then not minimally presented. 

Suppose $c=1$. Since $L'$ has an even number of tangles, $L'=\pm(2,-2,\dots, 2,-4)$. Thus up to mirroring,
\[
	K=(1, m_1, \dots, m_{2p-1}, -4)
\] 
where $m_i<-2$ for $1\leq i \leq 2p-1$, $i$ odd and $m_i>2$ for $1<i<2p-1$, $i$ even. Isotope $K$ according to Figure~\ref{fig:isotopies}.
\begin{figure}[htb]
	\labellist
	\small \hair 2pt
	\pinlabel $m_1,\dots,m_{2p-1}$ at 160 100
	\pinlabel $m_1,\dots,m_{2p-1}$ at 545 100
	\endlabellist
\centering
\includegraphics[height=35mm]{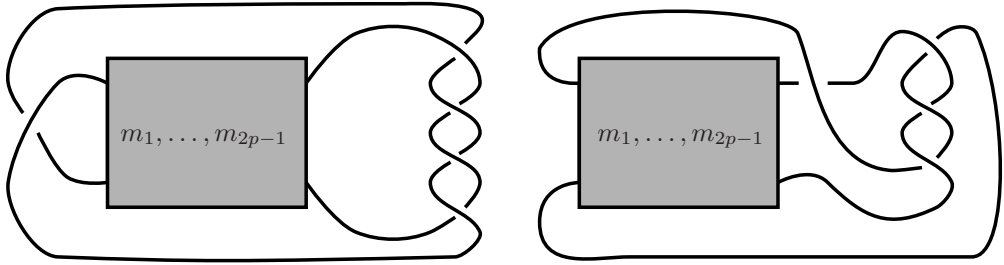}
\caption[The isotopy performed on the Type 2B knot $K$]{The isotopy performed on the Type 2B knot $K=(1, m_1, \dots, m_{2p-1}, -4)$.}
\label{fig:isotopies}
\end{figure}
\begin{figure}[ht]
	\labellist
	\small \hair 2pt
	\pinlabel $-1$ at 350 280
	\pinlabel $e_1$ at 400 300
	\pinlabel $-1$ at 320 110
	\pinlabel $-1$ at 385 110
	\pinlabel $-1$ at 445 110
	\pinlabel $e_2$ at 375 80	
	\pinlabel $e_3$ at 430 80	
	\pinlabel $e_4$ at 290 80			
	\pinlabel $\vdots$ at 102 215
	\pinlabel $\vdots$ at 8 215
	\pinlabel $\vdots$ at 230 215	
	\pinlabel $\dots$ at 170 215
	\endlabellist
\centering
\includegraphics[height=55mm]{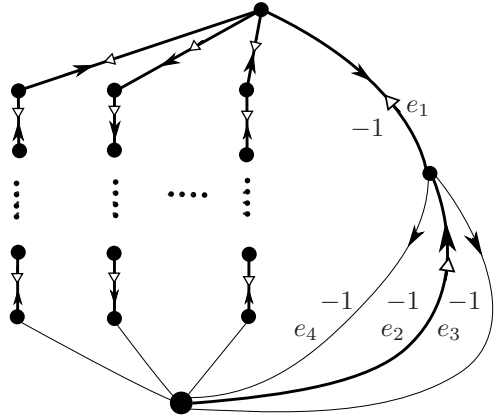}
\caption{The black graph after isotoping the Type 2B knot $(1, m_1, \dots, m_{2p-1}, -4)$ with $T_{\xmin}$ in bold. }
\label{fig:type2b}
\end{figure}
After this isotopy, the knot diagram admits a black graph whose edges are all labeled $\eta(e)=\pm1$, and a white graph where all of the edges are labeled $0$. Thus we only need to consider maximal trees of the black graph to compute $\Delta_K(t)$. This is no longer a pretzel presentation, but as can be seen in Figure \ref{fig:type2b} we can make sense of the terms trunk, top vertex, trade, etc. and may apply the content of Section \ref{subsec:countinglemmas} in an analogous manner.
\begin{claim}
\label{claim:minunique2}
After isotopy, orient $K$ so that the strands of the first tangle point upward. Then there is a unique state $\xmin$ with minimal $A$-grading.
\end{claim}
\begin{proof}[Proof of Claim~\ref{claim:minunique2}]
Refer to Figures \ref{fig:isotopies} and \ref{fig:type2b}. Since $m_i<-2$ for $i$ odd and $m_i>2$ for $i$ even, then for all edges $e\in T(m_i)\subset G_B$, $i=1,\dots 2p-1$,
\[
	\eta(e) =\left\{\begin{array}{ll}
	+1 & e\in T(m_i), i  \text{ odd} \\
	-1& e\in T(m_i), i \text{ even}. 
	\end{array}\right. 
\]
There are four additional edges in $G_B$, and each is labeled $\eta(e)=-1$. Let $\xmin$ be the state with trunk $e_1\cup e_2$ and with no other edges incident to the black root. For all $e\in T_{\xmin}$ with $e\neq e_1$, $\sigma(e)\eta(e)=-1$, and for $e_1$, $\sigma(e_1)\eta(e_1)=+1$. Because $|m_i|>2$ for $i=1,\dots, 2p-1$, then for any other state, the corresponding $A$-grading is strictly greater than $A(\xmin)$. Hence $\xmin$ is the unique state with minimal $A$-grading. 
\end{proof}
It is easy to verify that there are exactly $2p+1$ states in $A$-grading $-g(K)+1$, all of which are obtained by trades along any of $m_1,\dots , m_{2p-1}$ or by replacing $e_2$ with $e_3$ or $e_4$. Each of these $2p+1$ states is supported in the same $M$-grading, by an argument similar to Lemma~\ref{lem:trades}. Hence, $|a_{-g(K)+1}| = 2p+1$ whenever $K$ contains any tangle of length one, thus completing the proof of Lemma~\ref{lem:type2b} in this case.

Suppose now that there are no tangles of length one in $K$. Since $L'$ is a fibered Type 1 link, $L'$ is isotopic to $\pm(2,-2,2, -2,\dots, 2,-2,n)$, for some $n\in\Z$.  Since $K$ is Type 2B, up to mirroring there exists a permutation of the tangles such that the resulting knot, denoted $K^\tau$, is of the form
\[
	K^\tau=(m_1, \dots, m_{2p}, \bar{m})
\]
where $m_i>0$ when $i$ is odd, $m_i<0$ is negative when $i$ is even, and $\bar{m}$ is even. Since $K^\tau$ has no tangles of length one, the auxiliary link for $K^\tau$ is isotopic to $\pm(2,-2,2, -2,\dots, 2,-2,n^\tau)$, for some $n^\tau\in\Z$, and therefore $K^\tau$ is a Type 2B fibered pretzel knot. Because $K^\tau$ is a fibered mutant of the fibered knot $K$, it shares the same Alexander polynomial and genus. Therefore it suffices to work with $K^\tau$.
\begin{figure}[h]
	\labellist
	\tiny \hair 2pt
	\pinlabel $-1$ at 80 330
	\pinlabel $-1$ at 45 260
	\pinlabel $-1$ at 45 135
	\pinlabel $0$ at 65 70
	\pinlabel $+1$ at 170 320
	\pinlabel $+1$ at 155 260
	\pinlabel $+1$ at 155 135	
	\pinlabel $0$ at 140 70
	\pinlabel $-1$ at 215 320
	\pinlabel $-1$ at 235 260
	\pinlabel $-1$ at 235 135	
	\pinlabel $0$ at 245 70
	\pinlabel $+1$ at 280 300
	\pinlabel $+1$ at 290 260
	\pinlabel $+1$ at 295 135	
	\pinlabel $0$ at 315 70
	\pinlabel $0$ at 355 70
	\pinlabel $0$ at 355 330
	\pinlabel $0$ at 385 260		
	\endlabellist
\centering
\includegraphics[height=55mm]{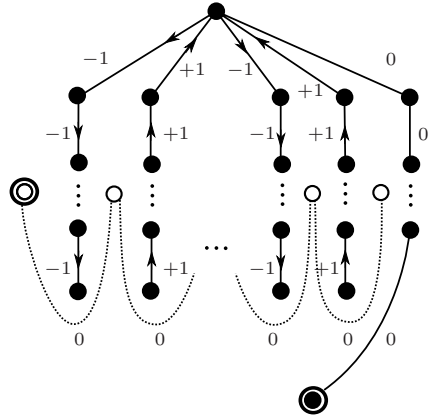}
\caption{The unique minimal state for a Type 2B knot $K^\tau$ with no tangles of length one. Labels $\eta(e)$ are indicated in the diagram.}
\label{fig:type2b-no1}
\end{figure}

When the pretzel diagram for $K^\tau$ is oriented so that the strands of $m_1$ point downward, $K^\tau$ admits a unique state $\xmin$ with minimal $A$-grading $-g(K^\tau)$. This state has trunk $T(\bar{m})$, and no other edges of $T_{\xmin}$ are incident to the root. See Figure~\ref{fig:type2b-no1}. Because the tangles alternate sign, every edge of $\tree_{\xmin}$ contributes $\sigma(e)\eta(e)=-1$ or $0$ to the sum for $A(\xmin)$. Because there are no tangles of length one, any other state will have a strictly greater $A$-grading. Hence $\xmin$ is unique and minimally $A$-graded. Moreover, every state supported in $A$-grading $-g(K^\tau)+1$ is a trade because there is a unique $\bar{m}$ and there are no tangles of length one. By Lemma~\ref{lem:trades}, there are $2p$ trades, each supported in $M$-grading $M(\xmin)+1$. Hence $|a_{-g(K^\tau)+1}|=2p \geq 2$, and this implies $|a_{-g(K)+1}|=2p \geq 2$. 
\end{proof}
%

\section{Type 3 knots} 
\label{sec:type3}
Each tangle in a Type 3 knot is an $m_i$, and therefore all edges $e\in G_B$ and $e^*\in G_W$ are labeled $\eta(e)=\pm1$ and $\eta(e^*)=0$, respectively (see Figure~\ref{fig:labels}).  In particular, the Alexander polynomials of Type 3 knots can be computed solely using the black graph $G_B$ and black maximal trees $T_\x$. Moreover, in this case, $K$ is a pretzel knot of even length, so we will assume $K$ has at least four tangles.

\subsection{Type 3-min}
\label{subsec:type3-min}
A Type 3-min knot $K$ has $p$ positive tangles and $p$ negative tangles. Of these there is a unique tangle of minimal length and an even tangle, which are possibly the same tangle. By assumption, since $K$ is fibered, $L'=\pm(2,-2,\dots, 2, -2)$ also has an even number of tangles, and thus by uniqueness of the minimal tangle, there are no tangles of length one.
\begin{lemma}
\label{lem:type3-min}
For all fibered pretzel knots of Type 3-min other than $K=\pm(3,-5,3,-2)$, there exists a coefficient of the Alexander polynomial such that $|a_s|\geq 2$.
\end{lemma}
\begin{proof}
By the conditions on $L'$, the tangles of $K$ alternate sign. After mirroring and cyclic permutation, we may assume $n_i$ is positive when $i$ is odd, $n_i$ is negative when $i$ is even, and $|n_{2p}|$ is minimal. For all $e \in T(n_i)$, $\eta(e)=-1$ when $i$ is odd and $\eta(e)=+1$ when $i$ is even. Orient the pretzel diagram so that the first tangle points downward. Let $\xmin$ be the state with trunk $T(n_{2p})$ and no other edges incident to the root (see the example in Figure \ref{fig:type3-min-min}). Because the tangles alternate sign, $\eta(e)\sigma(e)=-1$ for all $e\in T(n_i)$ for $i=1, \dots, 2p-1$, and for $e\in T(n_{2p})$, $\eta(e)\sigma(e)=+1$. Since $n_{2p}$ is the unique minimal length tangle, $A(\xmin)$ is minimal and $\xmin$ is the unique state with minimal $A$-grading. 
 
\begin{figure}[t]
	\labellist
	\tiny \hair 2pt
	\pinlabel $-1$ at 25 230
	\pinlabel $-1$ at 20 190
	\pinlabel $-1$ at 20 155
	\pinlabel $-1$ at 20 120
	\pinlabel $+1$ at 70 225
	\pinlabel $+1$ at 60 190
	\pinlabel $+1$ at 60 155	
	\pinlabel $+1$ at 60 120
	\pinlabel $+1$ at 60 90	
	\pinlabel $+1$ at 60 55
	\pinlabel $-1$ at 120 215
	\pinlabel $-1$ at 125 190
	\pinlabel $-1$ at 125 155	
	\pinlabel $-1$ at 125 120
	\pinlabel $+1$ at 200 235
	\pinlabel $+1$ at 230 190
	\pinlabel $+1$ at 230 120
	\pinlabel $+1$ at 205 60		
	\endlabellist
\centering
\includegraphics[height=60mm]{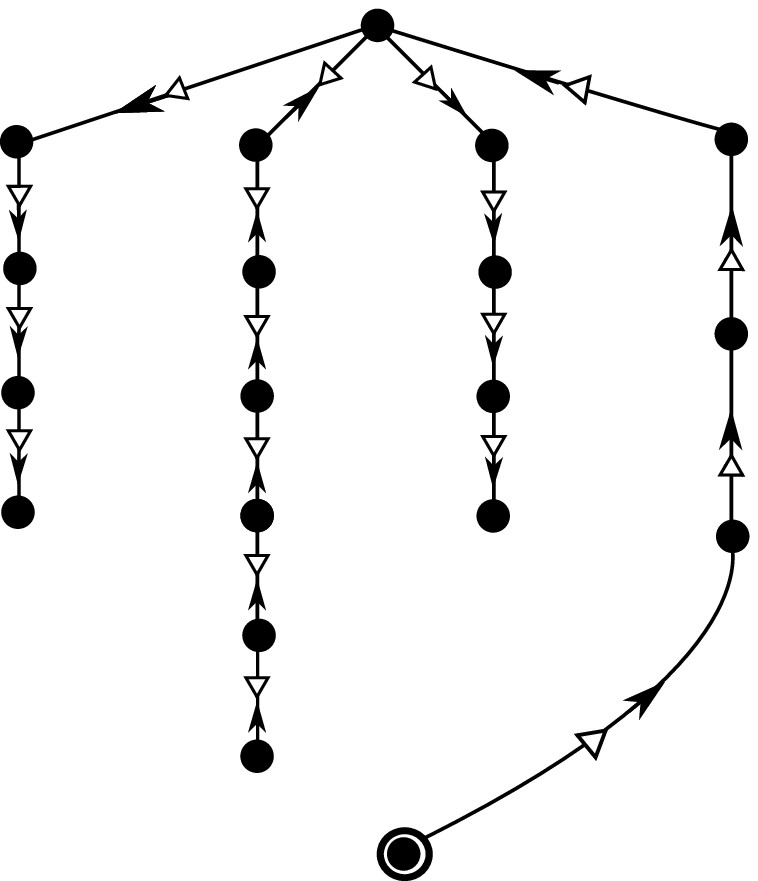}
\caption{An example of the unique minimal state for the Type 3-min knot $(5,-7,5,-4)$, with trunk along the unique tangle of minimum length.}
\label{fig:type3-min-min}
\end{figure}

By Lemma~\ref{lem:trades}, there are $2p-1$ trades in bigrading $(-g(K)+1,M(\xmin)+1)$. Since there is no $m_{ij}$, all states in $A$-grading $-g(K)+1$ which are not trades have a corresponding black tree with trunk $T(n_j)$ such that $|n_j| = |n_{2p}|+1$, by an argument similar to that in Lemma~\ref{lem:type2a-2}. Denote such a state by $\x_j$.  First suppose $n_{2p}$ is odd. Since there is exactly one even tangle, there is at most one state $\x_j$. If no such $\x_j$ exists $|a_{-g(K)+1}|=2p-1$. Otherwise $|a_{-g(K)+1}|=(2p-1)\pm 1$, depending on $M(\x_j)$. Since $2p\geq 4$, we have $|a_{-g(K)+1}|\geq 2$. Now suppose $n_{2p}$ is even. 
\begin{claim}
\label{claim:mgradings2}
Let $n_j$ be a tangle of length $|n_{2p}|+1$. Then,
\[
	M(\x_j) =\left\{
		\begin{array}{ll}
			M(\xmin) & j \text{ odd} \\
			M(\xmin) +1 & j \text{ even.}
		\end{array} \right.
\]
\end{claim}
\begin{proof}[Proof of Claim~\ref{claim:mgradings2}]
Fix $j$ such that  $|n_j|=|n_{2p}|+1$. Recall that the trunk of $T_{\xmin}$ is $T(n_{2p})$, the trunk of $T_{\x_j}$ is $T(n_j)$, and for all $e \in T(n_i)$, $\eta(e)=-1$ when $i$ is odd and $\eta(e)=+1$ when $i$ is even. Additionally, outside of $T(n_{2p})$ and $T(n_j)$, $T_{\xmin}$ and $T_{\x_j}$ agree. For $\xmin$ the values of $\sigma$ are given by
\[
	\sigma(e) =\left\{\begin{array}{ll}
	 +1 & e\in T_{\xmin}\cap T(n_{2p}) \\
	 +1 &  e\in T_{\xmin}\cap T(n_j), j \text{ odd} \\
	 -1 & e\in T_{\xmin}\cap T(n_j), j \text{ even,} \\
	\end{array}\right.
\]
and for $\x_j$ the values of $\sigma$ are given by
\[
	\sigma(e) =\left\{\begin{array}{ll}
	 -1 & e\in T_{\x_j}\cap T(n_{2p}), j \text{ odd or even }\\
	 -1 & e\in T_{\x_j}\cap T(n_j), j \text{ odd} \\
	 +1 &  e\in T_{\x_j}\cap T(n_j), j \text{ even.} 
	\end{array}\right.
\]
Suppose $j$ is odd. Then because $|n_j| = |n_{2p}|+1$,
\begin{equation*}
	M(\x_j)-M(\xmin) 
		= \sum_{	
			\begin{tiny}\begin{array}{c}
				e \in T_{\x_j} \\
				\sigma(e)=1
			\end{array} \end{tiny}
		} \hspace{-5mm} \eta(e) 
		- \hspace{-3mm} \sum_{	
			\begin{tiny}\begin{array}{c}
				e \in T_{\xmin} \\
				\sigma(e)=1
			\end{array} \end{tiny}
		}\hspace{-5mm} \eta(e) 
		=  -\big( |n_{2p}|-(|n_j|-1) \big)  = 0.
\end{equation*}
Suppose $j$ is even. Then
\begin{equation*}
	M(\x_j)-M(\xmin) 
		= \sum_{	
			\begin{tiny}\begin{array}{c}
				e \in T_{\x_j} \\
				\sigma(e)=1
			\end{array} \end{tiny}
		} \hspace{-5mm} \eta(e) 
		- \hspace{-3mm} \sum_{	
			\begin{tiny}\begin{array}{c}
				e \in T_{\xmin} \\
				\sigma(e)=1
			\end{array} \end{tiny}
		}\hspace{-5mm} \eta(e) 
		=  |n_j| -  |n_{2p}|   =1. \qedhere
\end{equation*}
\end{proof}
\begin{figure}[h]
	\labellist
	\footnotesize \hair 2pt
	\pinlabel $2n+1$ at 50 27
	\pinlabel $-q$ at 150 120
	\pinlabel $2n+1$ at 210 50
	\pinlabel $-2n$ at 250 100
	\pinlabel $2n$ at 420 90
	\pinlabel $-q+1$ at 530 90
	\pinlabel $2n$ at 660 90
	\pinlabel $-2n+1$ at 767 90		
	\pinlabel $\simeq$ at 340 180
	\endlabellist
\centering
\includegraphics[height=45mm]{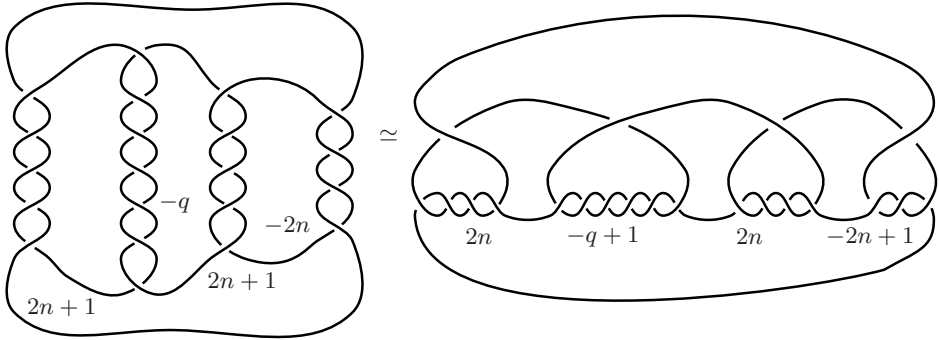}
\caption{The isotopy performed on $(2n+1, -q, 2n+1, -2n)$ to obtain a Seifert surface with reduced genus.}
\label{fig:isotopy3min-real}
\end{figure}
Let $E$ and $O$ be the number of states $\x_j$ with $j$ even and odd, respectively, and recall that there are $2p-1$ trades supported in $M$-grading $M(\xmin)+1$. By Claim~\ref{claim:mgradings2},
\[
	|a_{-g(K)+1}| =|(2p-1)+E-O |.
\]
Therefore $|a_{-g(K)+1}| \geq p -1$, so whenever $p>2$ we are done. 

The case $p=2$ remains. In particular, $|a_{-g(K)+1}| =|(2p-1)+E-O | \leq 1$ only when $E=0$ and $O=2$. Thus it suffices to consider
\[
	K =(2n+1, -q, 2n+1, -2n)
\]
where $n\geq1$ and $q\geq2n+3$ is odd. We will reduce the genus of the surface obtained by the Seifert algorithm by performing a particular isotopy of $K$, which is described in~\cite{Gabai:Detecting} and pictured in Figure~\ref{fig:isotopy3min-real}. Applying the Seifert algorithm to the new diagram gives a lower genus Seifert surface $F$ for $K$, suitable to apply Lemma \ref{lem:det}, but not necessarily a genus minimizing Seifert surface. We obtain
\[
g(F)=\frac{1}{2}(6n+q-3).
\]
By Equation~\ref{eqn:det}, 
\[
	det(K) = |4n(2n+1)q - (2n+1)^2q - 2n(2n+1)^2  |.
\]
In general, $det(K)>2g(F)+1 \geq 2g(K)+1$ is satisfied whenever 
\begin{equation}
\label{eqn:inequality}
	\big( 4n(2n+1) - (2n+1)^2 - 1 \big)q > 2n(2n+1)^2 +6n -2,
\end{equation}
and since $q \geq 2n+3$, this inequality holds for all $n>3$. Moreover, if $n=1$, $n=2$, or $n=3$, then $det(K)>2g(F)+1\geq 2g(K)+1$ whenever $q\geq 13$, $q\geq9$, or $q\geq11$, respectively. 
The only pairs $(n,q)$ not satisfying the inequality (\ref{eqn:inequality}) are: $(3,9), (2,7), (1,11),(1,9), (1,7),$ and $(1,5)$. The Alexander polynomials for the knots corresponding to the first five pairs are:
\begin{eqnarray*}
	&\Delta_{(7,-9,7,-6)} &= t^{-5} -t^{-4}+2t^{-2} -3t^{-1}+3 -3t+2t^2-t^4+t^5 \\
	&\Delta_{(5,-7,5,-4)} &= t^{-5} -t^{-4}+t^{-2} -2t^{-1}+3 -2t+t^2-t^4+t^5 \\
	&\Delta_{(3,-11,3,-2)} &= t^{-6} -t^{-5} +2t^{-3} -3t^{-2} +3t^{-1} -3 +3t-3t^2+2t^3-t^5+t^6	\\
	&\Delta_{(3,-9,3,-2)} &= t^{-7} - t^{-6} + t^{-4} - 2t^{-3} +3t^{-2} -4t^{-1} +5 -4t +3t^2 -2t^3 + t^4 -t^6 +t^7 \\
	&\Delta_{(3,-7,3,-2)} &= t^{-4} - t^{-3} +2t^{-1} -3 +2t -t^3 + t^4.
\end{eqnarray*}
Clearly each polynomial has some coefficient with $|a_s|>1$. The last pair of integers corresponds to $K=(3,-5,3,-2)$, the knot exempted in the statement of the lemma.
\end{proof}
The Alexander polynomial of $K$,
\[
	\Delta_K(t) = t^{-3}-t^{-2} +1 -t^2 + t^3,
\]
does not obstruct $K$ from admitting an $L$-space surgery. Therefore, we compute the knot Floer homology of $K$ in Table~\ref{table:hfkpretzel} using the Python program for $\hfk$ with $\F_2$ coefficients by Droz~\cite{Droz:Program} to observe directly that there exist Alexander gradings $s$ such that $\dim \hfk(K,s;\F_2)\geq 2$.  This implies that for these Alexander gradings, $\hfk(K,s) \not \cong 0$ or $\mathbb{Z}$.  Therefore, $K$ is not an $L$-space knot.  This completes the proof of Theorem~\ref{thm:pretzels} for Type 3-min pretzel knots.  
\begin{table}
	$\bgroup\arraycolsep=6pt 
		\begin{array}{|r|rrrrrrr|}\hline
			\multicolumn{8}{|c|}{\hfk(3,-5,3,-2) } \\ \hline
			&-3 & -2 & -1 & 0 & 1 & 2 & 3 \\ \hline
			4& & & & &  & & \F \\
			3& & & & &  &\F^3 &\\
			2& & & & & \F^4 & \F^2& \\
			1& & & & \F^3& \F^4 & &  \\
			0& & &\F^4 &\F^4  & &&  \\
			-1& & \F^3 & \F^4 & & & & \\
			-2& \F& \F^2 & & & & &  \\ \hline 
		\end{array} 
	\egroup$  \\
	\caption{The knot Floer homology groups of the knot $(3,-5,3,-2)$ are displayed with Maslov grading on the vertical axis and Alexander grading on the horizontal axis.}
	\label{table:hfkpretzel}	
\end{table}	

\subsection{Type 3-2A}
\label{subsec:type3-2a}
After mirroring, we may assume that for pretzel knots of Type 3-2A, there are $p+2$ positive tangles and $p$ negative tangles, and that of these $2p+2$ tangles, there is exactly one even tangle. Note that the property of being a Type 3-2A fibered pretzel knot does not change under mutation. 
\begin{lemma}
\label{lem:type3-2a}
Let $K$ be as above. If $K$ does not have exactly $p$ negative tangles of length one, $|a_{-g(K)+1}|\geq 2$.
\end{lemma}
\begin{proof}
Up to mutation, we may assume that $n_i$ is positive when $i$ is odd and that $n_i$ is negative when $i$ is even, except $n_{2p+2}$, which is positive. In $G_B$, $e\in T(n_i)$ is labeled $\eta(e)=-1$ for $i$ odd or $i=2p+2$ and $\eta(e)=+1$ for $i$ even, $i \neq 2p+2$. Orient $K$ so that the strands of the first tangle point downward. Then there is a unique state $\xmin$ with minimal $A$-grading represented by a black tree with trunk $T(n_{2p+2})$, as in Lemma~\ref{lem:type2a-2}. In particular, for all $e\in T_{\xmin}$, $\sigma(e)=+1$ if $e\in T(n_i)$ for $i$ odd or $i=2p+2$ and $\sigma(e)=-1$ if $i$ even, $i \neq 2p+2$. Every edge in $T_{\xmin}$ contributes $\eta(e)\sigma(e)=-1$ to the sum for $A(\xmin)$, so $\xmin$ is clearly minimally graded. It is unique because in any other tree there will be an edge contributing $\sigma(e)\eta(e)=+1$ to the $A$-grading. 

There are $2p-\ell+1$ trades in bigrading $(-g(K)+1, M(\xmin)+1)$ by Lemma~\ref{lem:trades}, where $\ell$ is the number of tangles of length one not counting the trunk. There are precisely $\ell$ other states in $A$-grading $-g(K)+1$. Each of these additional states, denoted $\x_j$, corresponds to a tangle $n_j$ of length one, as obtained in Lemma~\ref{lem:type2a-2}. Then,
\[
	M(\x_j) =\left\{\begin{array}{ll}
	M(\xmin)+1 & j \text{ is odd} \\
	M(\xmin)+2 & j\neq 2p+2 \text{ is even,}
	\end{array}\right.
\]
as in Claim~\ref{claim:mgradings1}.
If the length one tangles are positive (i.e. each $j$ is odd), then 
\[
	|a_{-g(K)+1}| = (2p-\ell+1) +\ell = 2p+1 > 2,
\]
and we are done. If the length one tangles are negative, then
\[
	|a_{-g(K)+1}| = (2p -\ell +1) -\ell  > 1 \Longleftrightarrow \ell < p. 
\] 
This verifies the statement of Lemma~\ref{lem:type3-2a}.
\end{proof}
The next lemma will complete the proof of Theorem~\ref{thm:pretzels} for Type 3-2A pretzel knots.
\begin{lemma}
Let $K$ be a Type 3-2A knot with exactly $p$ negative length one tangles, and $p+2$ positive tangles. Then there exists some coefficient $a_s$ of $\Delta_K(t)$ with $|a_s|>1$.
\end{lemma}
\begin{proof}
After reindexing the tangles,
\[
	K = (\underbrace{-1,\dots, -1}_{p}, w_1, \dots, w_{p+2}),
\]
where there exists some $i$ such that $w_i \geq4$ is even (since $K$ is minimally presented, $w_i\neq 2$ for any $i$) and for all other $i$, $w_i\geq 3$ is odd. By Theorem~\ref{thm:fiberedpretzels}, the genus of $K$ is obtained by applying the Seifert algorithm to the standard projection,
\begin{equation*}
	g(K)=\frac{1}{2}\big( \sum_{i=1}^{p+2}(w_i-1) + 1 \big).
\end{equation*}
Let $W=w_1\cdots w_{p+2}$. Using Equation~\ref{eqn:det},
\begin{eqnarray*}
	det(K) &=& \big| W \big(-p+ \sum^{p+2}_{i=1}\frac{1}{w_i} \big) \big| \\
	 &\geq&   \big|  W(p - \frac{1}{4} -\sum^{p+1}_{i=1}\frac{1}{3}   ) \big|  \\
	 &\geq&   W\cdot \frac{8p-7}{12}.
\end{eqnarray*}
Whenever $p\geq 2$, we have
\[
	det(K) > ( \sum_{i=1}^{p+2}w_i) -p =  2g(K)+1.
\]
Now apply Lemma \ref{lem:det}. If $p=1$ then $K= (-1, w_1, w_2, w_3)$. Now suppose one of the $w_i$ is at least five. Then,
\begin{eqnarray*}
	det(K) &=& \big| W \big(1- \sum^{3}_{i=1}\frac{1}{w_i} \big) \big| \\
	&\geq&  W\cdot \frac{13}{60}  \\
	&>& (\sum_{i=1}^{3}w_i ) -1 \\
	&=&  2g(K)+1.
\end{eqnarray*}
The only Type 3-2A fibered pretzel knot with four or more strands which has not been addressed is $K=(-1,3,3,4)$, which has Alexander polynomial
\[
	\Delta_{(-1,3,3,4)}= t^{-4} - t^{-3} + 2t^{-1} - 3 +2t -t^3 +t^4.
\]
Clearly there exist coefficients with $|a_s|>1$.
\end{proof}
\subsection{Type 3-2B}
\label{subsec:type3-2b}
Let $K$ be a fibered Type 3-2B pretzel knot. There are $p$ positive tangles, and $p$ negative tangles. By assumption the auxiliary link $L'$ is not isotopic to $\pm(2,-2,\dots,2,-2)$, and $K$ is fibered if and only if $L'$ is fibered. There are no tangles of $L'$ equal to $\pm 1$ and therefore $L'$ cannot be of Type 1-(1). Since there is no $m_{ij}$, there are no tangles equal to $\pm4$, and so we may also rule out Type 1-(3). Therefore $L'$ must fall into the Type 1-(2) subcase of Type 1 knots, which are of the form $\pm(2,-2,\dots, 2,-2, n)$, where $n\in \Z$. This can only happen if $n=\pm 2$ and $K$ contains a unique tangle of length one. 

Up to mirroring and isotopy, $K= (n_1\dots, n_{2p})$, where $n_i$ is positive for $i$ odd, negative for $i$ even, and $n_{2p}=-1$. Orient $K$ so that the strands of the first tangle point downward. Then $\eta(e)=-1$ when $e\in T(n_i)$ for $i$ odd, and $\eta(e)=+1$ when $e\in T(n_i)$ for $i$ even. As in the proof of Lemma~\ref{lem:type3-min}, there exists a state $\xmin$ with minimal $A$-grading with trunk $T(n_{2p})$, and with the property that $\sigma(e)=+1$ when $e\in T(n_i)$ for $i$ odd, and $\sigma(e)=-1$ when $e\in T(n_i)$ for $i$ even. The only possible states which are not trades must occur along tangles of length two. Since there is a single even tangle, there is at most one such state. By Equation \ref{eqn:euler} and Lemma \ref{lem:trades} this implies that $|a_{-g(K)+1}|$ is at least $2p-2$, and hence $|a_{-g(K)+1}|\geq 2$.

This completes the case analysis required to prove Theorem \ref{thm:pretzels}. 
\begin{footnotesize}
	\bibliographystyle{alpha}
	\bibliography{bibliography}
\end{footnotesize}

\end{document}